\documentclass[12pt, a4paper]{article}
\usepackage{amsfonts}
\usepackage{mathrsfs}
\usepackage{latexsym}
\usepackage{xy}
\usepackage{amsfonts,amsmath,amssymb,amsthm}
\usepackage{color}

\xyoption{all}

\newcommand{\bcen}{\begin{center}}     \newcommand{\ecen}{\end{center}}
\newcommand{\bay}{\begin{array}}      \newcommand{\eay}{\end{array}}
\newcommand{\beq}{\begin{eqnarray*}}      \newcommand{\eeq}{\end{eqnarray*}}

\def\rad{\mathrm{rad}}

\def\Hom{\mathrm{Hom}}
\def\Ker{\mathrm{Ker}}
\def\Im{\mathrm{Im}}

\def\mod{\mathrm{mod}}
\def\Mod{\mathrm{Mod}}
\def\id{\mathrm{id}}

\def\lim{\mathrm{lim}}

\def\per{\mathrm{per}}
\def\proj{\mathrm{proj}}
\def\Gproj{\mathrm{Gproj}}

\def\add{\mathrm{add}}

\begin{document}

\newtheorem{theorem}{Theorem}[section]
\newtheorem{proposition}[theorem]{Proposition}
\newtheorem{lemma}[theorem]{Lemma}
\newtheorem{corollary}[theorem]{Corollary}
\newtheorem{remark}[theorem]{Remark}
\newtheorem{example}[theorem]{Example}
\newtheorem{definition}[theorem]{Definition}
\newtheorem{question}[theorem]{Question}
\numberwithin{equation}{section}

\title{\large\bf
A note on singularity categories and triangular matrix algebras}

\author{\large Yongyun Qin}
\date{\footnotesize School of Mathematics, Yunnan Normal University, \\ Kunming, Yunnan 650500, China. E-mail:
qinyongyun2006@126.com
}

\maketitle

\begin{abstract} Let $\Lambda =
\left[\begin{array}{cc} A & 0 \\ M & B  \end{array}\right] $
be an Artin algebra and $_BM_A$ a $B$-$A$-bimodule.
We prove that there is a triangle equivalence $D_{sg}(\Lambda) \cong D_{sg}(A)\coprod D_{sg}(B)$
between the corresponding singularity categories if
$_BM$ is semi-simple and $M_A$ is projective. As a result, we
obtain a new method for describing the singularity categories of certain bounded quiver algebras.
\end{abstract}

\medskip


\medskip


\bigskip

\section{\large Introduction}

\indent\indent
Throughout $A$ is an Artin algebra, and $\mod A$ is the category of finitely generated
right $A$-modules. Following \cite{Buch21},
the {\it singularity category} $D_{sg}(A)$ of
$A$ is the Verdier quotient of the bounded derived category of finitely generated modules over $A$ by the
full subcategory of perfect complexes. It measures the homological singularity of an algebra $A$ in the sense
that $A$ has finite global
dimension if and only if  $D_{sg}(A)=0$. Moreover, the singularity category captures the stable homological
features of an algebra \cite{Buch21}. It is known
that the singularity category of
a Gorenstein algebra can be characterized by the stable category of Gorenstein projective (also called maxima Cohen-Macaulay) $A$-modules
\cite{Buch21, Hap91}.
Recently, many people are trying to describe the singularity categories for some special classes of algebras,
such as Nakayama algebras \cite{CY14, Rin13, Shen15}, gentle algebras \cite{Kal15}, Gorenstein monomial algebras
\cite{CSZ18, LZ21}, quadratic monomial algebras \cite{Chen18}
and algebras with radical square zero \cite{Chen11}.

An approach to investigate singularity categories is to
compare these categories between algebras related to each another.
This was studied in the context of idempotent reduction \cite{Chen09, PSS14, Shen21}, arrow removal
\cite{EPS22}, homological ideal \cite{Chen14}, simple gluing algebra \cite{Lu19},
triangular matrix algebra \cite{Chen09, LL15, Lu17, PSS14, Z13} and
recollement \cite{LL15, Qin20}. In particular, Lu proved that if $\Lambda =
\left[\begin{array}{cc} A & 0 \\ M & B  \end{array}\right]$ is a finite-dimensional
triangular matrix algebra over a field
and $_BM_A$ is projective as a $B$-$A$-bimodule,
then the singularity category and the Gorenstein defect category of $\Lambda$
are equivalent to the direct sum of those of $A$ and $B$ (see the proof of \cite[Proposition 4.2 and Theorem 4.4]{Lu17}).
Recall that the Gorenstein defect category is a Verdier quotient of the singularity category by modulo
the isomorphic image of the stable category of Gorenstein projective modules, and
in general, a triangle equivalence on singularity category can not produce such an equivalence on Gorenstein defect category.
So the result of Lu motivates
the following natural questions: Can we weaken the hypotheses of Lu
to just get the desired equivalence on singularity category, but not on Gorenstein defect category?
or, under what conditions do we have an equivalence
$D_{sg}(\Lambda) \cong D_{sg}(A)\coprod D_{sg}(B)$? We answer the second question by the following theorem,
which is listed as Theorem~\ref{theorem-main} in this paper.

\medskip

{\bf Theorem I.}  {\it Let $\Lambda =
\left[\begin{array}{cc} A & 0 \\ M & B  \end{array}\right] $
be an Artin algebra and $_BM_A$ a $B$-$A$-bimodule. If
$_BM$ is semi-simple and $M_A$ is projective,
then there is a triangle equivalence $D_{sg}(\Lambda) \cong D_{sg}(A)\coprod D_{sg}(B)$.}

\medskip

Theorem I can be used to reduce the singularity categories of certain bounded quiver algebras.
Let $\Lambda = kQ/I$ be a bounded quiver algebra
over a field $k$. Suppose there exist two disjoint full subquivers
of $Q$, named as $\Gamma$ and $\bar{\Gamma}$ such that: (1) $Q_0=\Gamma_0\cup \bar{\Gamma}_0$;
(2) there is at least one arrow from $\bar{\Gamma}_0$ to $\Gamma_0$,
and there are no arrows from $\Gamma_0$ to $\bar{\Gamma}_0$; (3) $I$ is generated
by $I\cap k\Gamma$, $I\cap k\bar{\Gamma}$ and $\alpha\beta$, where
$\alpha$ is an arrow in $\bar{\Gamma}_1$ and $\beta$ is an arrow form $\bar{\Gamma}_0$ to $\Gamma_0$.
Then $D_{sg}(\Lambda) \cong D_{sg}(A)\coprod D_{sg}(B)$ with $A=k\Gamma/(I\cap k\Gamma)$ and
$B=k\bar{\Gamma}/(I\cap k\bar{\Gamma})$. For more detail, see Corollary~\ref{cor-quiver} and Example~\ref{exam-1}.

We remark that the above reduction situation was considered
in \cite{BM19, BM22} with respect to the $\phi$-dimension,
and a similar result known for singularity category is
the arrow removal operation. That is, let $\Lambda$ be as above and
all conditions hold except we replace (3) by $I=\langle I\cap k\Gamma$, $I\cap k\bar{\Gamma}\rangle$,
then there is a triangle equivalence $D_{sg}(\Lambda) \cong D_{sg}(A)\coprod D_{sg}(B)$
which induces an equivalence between the Gorenstein defect categories
\cite{EPS22, Qin22}. Hence, Corollary~\ref{cor-quiver} weaken the condition of \cite{EPS22}
if we only consider the equivalence on singularity category, see Remark~\ref{remark-1}.

This paper is organized as follows. In section 2, we recall some relevant
definitions and conventions.
In section 3 we prove Theorem I,
and in section 4, we use Theorem I to
reduce the singularity categories of certain bounded quiver algebras.

 \section{\large Definitions and conventions}\label{Section-definitions and conventions}

\indent\indent In this section we will fix our notations and recall some basic definitions.

Throughout, all the algebras are Artin algebras over a commutative Artinian ring $R$.
Let $A$ be such an algebra and let $\rad (A)$ be the
Jacobson radical of $A$.  We denote by $\Mod A$ the
category of right $A$-modules, and we view left $A$-modules
as right $A^{op}$-modules, where $A^{op}$ is the opposite algebra of $A$.
Denote by $\mod A$ and $\proj A$
the full subcategories of $\Mod A$ consisting of all finitely
generated modules and finitely
generated projective modules, respectively.
For $M \in \mod A$,
let $\rad (M)$ be the
Jacobson radical of $M$,
and $\add (M)$ the full
additive subcategory of $\mod A$ consisting of all direct summands of finite sums of copies of $M$.

Let
$K^b (\proj A)$  be the bounded homotopy category of complexes
over $\proj A$, and
$\mathcal{D}(\Mod A)$ (resp. $\mathcal{D}^b(\mod A)$)
the derived category (resp. bounded derived category) of complexes over $\Mod A$ (resp. $\mod A$).
We denote by $[-]$ the shift functor on complexes.

Usually, we
just write $\mathcal{D} A$ (resp. $\mathcal{D}^b(A)$) instead of $ \mathcal{D}(\Mod A)$ (resp. $ \mathcal{D}^b(\mod A)$).
Up to isomorphism, the objects in $K^{b}(\proj A)$ are
precisely all the compact objects in $\mathcal{D} A$. For
convenience, we do not distinguish $K^{b}(\proj A)$ from the {\it
perfect derived category} $\mathcal{D}_{\per}(A)$ of $A$, i.e., the
full triangulated subcategory of $\mathcal{D} A$ consisting of all
compact objects, which will not cause any confusion. Moreover, we
also do not distinguish
$\mathcal{D}^b(A)$ from its essential image under the
canonical embedding into $\mathcal{D} A$.

Following \cite{Buch21, Orl04}, the {\it singularity category} of $A$ is the
Verdier quotient $D_{sg}(A) = \mathcal{D}^b(A)/K^{b}(\proj A)$.
A finitely
generated $A$-module $M$ is called
{\it Gorenstein projective} if there is an
exact sequence $$\xymatrix{P^\bullet = \cdots \ar[r]& P^{-1}
\ar[r]^{d^{-1}} & P^0 \ar[r]^{d^{0}} & P^1
\ar[r] & \cdots} $$
of $\proj A$ with $M= \Ker d^0$ such that $\Hom _A(P^\bullet, Q)$ is exact for
every $Q \in \proj A$. Denote by $\Gproj A$
the subcategory of $\mod A$ consisting of Gorenstein projective modules.
It is well known that $\Gproj A$
is a Frobenius category, and hence its stable category $\underline{\Gproj} A$ is a triangulated category.
Moreover, there is a canonical triangle functor $F: \underline{\Gproj} A \rightarrow D_{sg}(A)$
sending a Gorenstein projective module to the corresponding stalk complex concentrated in degree zero
\cite{Buch21}, and
the Verdier quotient $D_{def}(A):=D_{sg}(A)/ \Im F$
is called the {\it Gorenstein defect
category} of $A$ \cite{BJO15}.

Let $\mathcal{T}_1$, $\mathcal{T}$ and $\mathcal{T}_2$ be
triangulated categories. A {\it recollement} of $\mathcal{T}$
relative to $\mathcal{T}_1$ and $\mathcal{T}_2$ is given by
$$\xymatrix@!=4pc{ \mathcal{T}_1 \ar[r]^{i_*=i_!} & \mathcal{T} \ar@<-3ex>[l]_{i^*}
\ar@<+3ex>[l]_{i^!} \ar[r]^{j^!=j^*} & \mathcal{T}_2
\ar@<-3ex>[l]_{j_!} \ar@<+3ex>[l]_{j_*}} \eqno {\rm (R)} $$
such that

(R1) $(i^*,i_*), (i_!,i^!), (j_!,j^!)$ and $(j^*,j_*)$ are adjoint
pairs of triangle functors;

(R2) $i_*$, $j_!$ and $j_*$ are full embeddings;

(R3) $j^!i_*=0$ (and thus also $i^!j_*=0$ and $i^*j_!=0$);

(R4) for each $X \in \mathcal {T}$, there are triangles

$$\begin{array}{l} j_!j^!X \rightarrow X  \rightarrow i_*i^*X  \rightarrow
\\ i_!i^!X \rightarrow X  \rightarrow j_*j^*X  \rightarrow
\end{array}$$ where the arrows to and from $X$ are the counits and the
units of the adjoint pairs respectively \cite{BBD82}.

An {\it left (resp. right) recollement} of $\mathcal{T}$
relative to $\mathcal{T}_1$ and $\mathcal{T}_2$ is the upper (resp. lower) two rows of (R)
such that $i^*,i_*,j_!$ and $j^!$ (resp. $i_*,i^!,j^!$ and $j_*$)
satisfy the above conditions \cite{BGS88,Koe91,Par89}.
The left (resp. right) recollement
also called upper (resp. lower) recollement in some literature.

Let $\mathcal{T}_1$, $\mathcal{T}$ and $\mathcal{T}_2$ be
triangulated categories, and $n$ a positive integer. An {\it
$n$-recollement} of $\mathcal{T}$ relative to $\mathcal{T}_1$ and
$\mathcal{T}_2$ is given by the following $n+2$ layers of triangle functors
$$\xymatrix@!=4pc{ \mathcal{T}_1 \ar@<+1ex>[r] \ar@<-3ex>[r]_\vdots & \mathcal{T}
\ar@<+1ex>[r]\ar@<-3ex>[r]_\vdots \ar@<-3ex>[l] \ar@<+1ex>[l] &
\mathcal{T}_2 \ar@<-3ex>[l] \ar@<+1ex>[l]}$$ such that every
consecutive three layers form a recollement \cite{BGS88, QH16}.

\section{Singularity categories and triangular matrix
algebras}
\indent\indent Let $A$ and $B$ be two Artin algebras, and $_BM_A$ a $B$-$A$-bimodule
with $_BM$ semi-simple and $M_A$ projective. Assume that $\Lambda =
\left[\begin{array}{cc} A & 0 \\ M & B  \end{array}\right] $
is a triangular matrix algebra, and $e_1 =
\left[\begin{array}{cc} 1 & 0 \\ 0& 0  \end{array}\right] $, $e_2=
\left[\begin{array}{cc} 0 & 0 \\ 0& 1  \end{array}\right] $ are two idempotents. Then it follows from \cite[Example 3.4]{AKLY17}
that there is a
$2$-recollement
$$\xymatrix@!=9pc{ \mathcal{D}A \ar@<+1.5ex>[r]|{i_*} \ar@<-4.5ex>[r] &
\mathcal{D}\Lambda \ar@<+1.5ex>[r]|{j^*} \ar@<-4.5ex>[r] \ar@<-4.5ex>[l] \ar@<+1.5ex>[l]|{i^!} &
\mathcal{D}B  \ar@<-4.5ex>[l] \ar@<+1.5ex>[l]|{j_*}},$$
where $i_*=- \otimes^L_{A} e_1\Lambda $,
$i^!=- \otimes_{\Lambda} \Lambda e_1$, $j^*=-\otimes _\Lambda \Lambda e_2$ and
$j_*=-\otimes _B^L \Lambda/\Lambda e_1 \Lambda$.
Since $_A(e_1\Lambda)=_A$$A$ and $_B(\Lambda/\Lambda e_1 \Lambda)=_B$$B$, we have that
$i_*\cong - \otimes_{A} e_1\Lambda $ and $j_*\cong -\otimes _B \Lambda/\Lambda e_1 \Lambda$, and thus
$i_*$ (resp. $j_*$) sends modules over $A$ (resp. $B$) to modules over $\Lambda$.
It is clear that $i_*$, $i^!$, $j_*$ and $j^*$ restrict to exact functors between module categories, and
then they induce four triangle functors on
bounded derived categories. By \cite[Lemma 2.9 (e)]{AKLY17},
we have that both $i_*$ and $j^*$ restrict to $K^b(\proj)$.
Further,
since $i^!\Lambda=\Lambda e_1\cong A\oplus M$ and $M_A$ is projective, we have that
$i^!$ restricts to $K^b(\proj)$, and so does $j_*$ (see \cite[Proposition 3.2]{AKLY17}).
Therefore, $i_*$, $i^!$, $j_*$ and $j^*$ induce four triangle functors on singularity
categories, which are denoted by $\widetilde{i}_*$, $\widetilde{i}^!$, $\widetilde{j}_*$ and $\widetilde{j}^*$ respectively,
and it follows from \cite[Lemma 2.3]{Lu17} that there is a
right recollement of singularity
categories
$$\xymatrix@!=9pc{ D_{sg}(A) \ar@<+1.5ex>[r]|{\widetilde{i}_*} &
D_{sg}(\Lambda) \ar@<+1.5ex>[r]|{\widetilde{j}^*} \ar@<+1.5ex>[l]|{\widetilde{i}^!} &
D_{sg}(B) \ar@<+1.5ex>[l]|{\widetilde{j}_*}}.$$

The following lemmas are crucial in our observation.

\begin{lemma}\label{lem-syz}
Let $i_*=- \otimes_{A} e_1\Lambda :\mod A\rightarrow \mod \Lambda$
and $j_*=-\otimes _B \Lambda/\Lambda e_1 \Lambda :\mod B\rightarrow \mod \Lambda$
be two functors between module categories. Then for any $X\in \mod A$ and $Y\in \mod B$, we have
$\Omega _\Lambda (i_*
(X))\in \Im i_*$ and $\Omega _\Lambda (j_*
(Y))\cong j_*(Y')\oplus P$, where $Y'\in \mod B$ and $P\in \proj \Lambda$.
\end{lemma}

\begin{proof}
It is known that a right $\Lambda$-module is identified with a triple $(X,Y,f)$,
where $X\in \mod A$, $Y\in \mod B$ and
$f:Y\otimes _BM\rightarrow X$ is a morphism of right $A$-modules. Clearly, the projective $\Lambda$-modules
are of the form $(P,0,0)$ and $(Q\otimes _BM,Q,\id)$, where $P\in \proj A$ and $Q\in \proj B$.
For any $X\in \mod A$ and $Y\in \mod B$, we have $i_*(X)=X\otimes_{A} e_1\Lambda\cong (X,0,0)$
and $j_*(Y)=Y\otimes _B \Lambda/\Lambda e_1 \Lambda\cong (0,Y,0)$.
Let $P_X$ be the projective cover of $X_A$, and $\Omega _A(X)$
the syzygy of $X_A$. Then we
have an exact sequence $$0\rightarrow (\Omega _A(X),0,0) \rightarrow (P_X,0,0) \rightarrow (X,0,0)\rightarrow 0.$$
Since $(P_X,0,0)\in \proj \Lambda$, we get that $\Omega _\Lambda (X,0,0)\cong (\Omega _A(X),0,0)$,
up to some direct summands of projective
$\Lambda$-modules. Therefore, $\Omega _\Lambda (i_*
(X))\cong (\Omega _A(X),0,0)= i_*(\Omega _A(X))\in \Im i_*$.

Let $Q_Y$ be the projective cover of $Y_B$, and $\Omega _B(Y)$ the syzygy of $Y_B$.
Then we have exact sequences
$$\xymatrix{0 \ar[r] & \Omega _B(Y)
\ar[r]^{i} & Q_Y \ar[r] & Y
\ar[r] & 0},$$ and
$$0\rightarrow (Q_Y\otimes _BM, \Omega _B(Y),i\otimes \id) \rightarrow (Q_Y\otimes _BM,Q_Y,\id) \rightarrow (0,Y,0)\rightarrow 0.$$
Since $(Q_Y\otimes _BM,Q_Y,\id)\in \proj \Lambda$, we get $\Omega _\Lambda (0,Y,0)\cong (Q_Y\otimes _BM, \Omega _B(Y),i\otimes \id)$.
Note that $\Im i\subseteq \rad (Q_Y)$, and then $\Im (i\otimes \id)\subseteq
\rad (Q_Y)\otimes _BM=(Q_Y\rad B)\otimes _BM\cong Q_Y\otimes _B \rad M$, which
is equal to zero since $_BM$ is semi-simple. Therefore, we conclude that $i\otimes \id=0$, and
then $$(Q_Y\otimes _BM, \Omega _B(Y),i\otimes \id)\cong (Q_Y\otimes _BM, 0,0)\oplus (0, \Omega _B(Y),0).$$
Since $M_A$ is projective and $Q_Y\in \add B$, we have that $Q_Y\otimes _BM\in \add M_A
\subseteq \proj A$, and then $(Q_Y\otimes _BM, 0,0)\in \proj \Lambda$. Let $P=(Q_Y\otimes _BM, 0,0)$
and $Y'=\Omega _B(Y)$. Then we get $\Omega _\Lambda (0,Y,0)\cong (Q_Y\otimes _BM, 0,0)\oplus (0, \Omega _B(Y),0)
=P\oplus j_*(Y')$, where $Y'\in \mod B$ and $P\in \proj \Lambda$.
\end{proof}

\begin{lemma}\label{lem-orth}
Keep the notations as above. Then $\Hom_{D_{sg}(\Lambda) }(\Im \widetilde{j}_*, \Im \widetilde{i}_*)=0$.
\end{lemma}
\begin{proof}
For any $X^\bullet \in D_{sg}(A)$, there exists some $n\in \mathbb{Z}$
and $X\in \mod A$ such that $X^\bullet \cong X[n]$ in $D_{sg}(A)$,
see \cite[Lemma 2.1]{Chen11}. Similarly, any object $Y^\bullet \in D_{sg}(B)$
is isomorphic to $Y[m]$, for some $m\in \mathbb{Z}$
and $Y\in \mod B$. Therefore, we have isomorphisms
\begin{align*}
	\Hom_{D_{sg}(\Lambda) }(\widetilde{j}_*(Y^\bullet), \widetilde{i}_*
(X^\bullet)) &\cong \Hom_{D_{sg}(\Lambda) }(\widetilde{j}_*(Y[m]), \widetilde{i}_*
(X[n])) \\
	&\cong \Hom_{D_{sg}(\Lambda) }(\widetilde{j}_*(Y), \widetilde{i}_*
(X)[n-m]) \\
	&\cong \Hom_{D_{sg}(\Lambda) }(j_*(Y), i_*
(X)[n-m]).
\end{align*}
Now we will claim $\Hom_{D_{sg}(\Lambda) }(\widetilde{j}_*(Y^\bullet), \widetilde{i}_*
(X^\bullet))=0$ under three cases.

\medskip

{\it Case 1.} $n-m=0$. Then it follows from \cite[Proposition 2.3]{Chen11}
or \cite[Example 2.3]{KV87} that $\Hom_{D_{sg}(\Lambda) }(j_*(Y), i_*
(X))\cong \underrightarrow{\lim}_{l\geq 0} \underline{\Hom}_\Lambda(\Omega ^l(j_*(Y)), \Omega^l(i_*
(X)))$. Clearly, $\underline{\Hom}_\Lambda(j_*(Y), i_*
(X))=\underline{\Hom}_\Lambda((0,Y,0), (X,0,0))=0$, and Lemma~\ref{lem-syz} implies that
$\underline{\Hom}_\Lambda(\Omega ^l(j_*(Y)), \Omega^l(i_*
(X)))\in \underline{\Hom}_\Lambda(\Im j_*, \Im i_*)=0,$ for any $l\geq 1$.
Therefore, we get that $\Hom_{D_{sg}(\Lambda) }(j_*(Y), i_*
(X))=0$.

\medskip

{\it Case 2.} $n-m>0$. Then we have
\begin{align*}
\Hom_{D_{sg}(\Lambda) }(j_*(Y), i_*
(X)[n-m]) &\cong \Hom_{D_{sg}(\Lambda) }(j_*(Y)[-n+m], i_*
(X)) \\
&\cong \Hom_{D_{sg}(\Lambda) }(\Omega _\Lambda ^{n-m}(j_*(Y)), i_*
(X)) \\
&\cong \underrightarrow{\lim}_{l\geq 0} \underline{\Hom}_\Lambda(\Omega ^{l+n-m}(j_*(Y)), \Omega^l(i_*
(X))).
\end{align*}
By Lemma~\ref{lem-syz}, we get $\underline{\Hom}_\Lambda(\Omega ^{l+n-m}(j_*(Y)), \Omega^l(i_*
(X)))\in \underline{\Hom}_\Lambda(\Im j_*, \Im i_*)$, which is equal to zero.
Therefore, $\Hom_{D_{sg}(\Lambda) }(j_*(Y), i_*
(X)[n-m])=0$.

\medskip

{\it Case 3.} $n-m<0$. Then we have $$\Hom_{D_{sg}(\Lambda) }(j_*(Y), i_*
(X)[n-m]) \cong \Hom_{D_{sg}(\Lambda) }(j_*(Y), \Omega _\Lambda ^{m-n}(i_*
(X))),$$ and we can prove $\Hom_{D_{sg}(\Lambda) }(j_*(Y), i_*
(X)[n-m])=0$ in the same way as Case 2.

\end{proof}

\begin{theorem}\label{theorem-main}
Let $_BM_A$ be a $B$-$A$-bimodule
with $_BM$ semi-simple and $M_A$ projective, and let $\Lambda =
\left[\begin{array}{cc} A & 0 \\ M & B  \end{array}\right] $.
Then $$D_{sg}(\Lambda) \cong D_{sg}(A)\coprod D_{sg}(B).$$
\end{theorem}

\begin{proof}
From the right recollement $$\xymatrix@!=9pc{ D_{sg}(A) \ar@<+1.5ex>[r]|{\widetilde{i}_*} &
D_{sg}(\Lambda) \ar@<+1.5ex>[r]|{\widetilde{j}^*} \ar@<+1.5ex>[l]|{\widetilde{i}^!} &
D_{sg}(B) \ar@<+1.5ex>[l]|{\widetilde{j}_*}},$$ we have that
$\Hom_{D_{sg}(\Lambda) }(\Im \widetilde{i}_*, \Im \widetilde{j}_*)=0$, and both $\widetilde{i}_*$
and $\Im \widetilde{j}_*$ are fully faithful.
On the other hand, it follows from Lemma~\ref{lem-orth} that $\Hom_{D_{sg}(\Lambda) }(\Im \widetilde{j}_*, \Im \widetilde{i}_*)$
$=0$,
and then the functor $(\widetilde{i}_*, \widetilde{j}_*):D_{sg}(A)\coprod D_{sg}(B)\rightarrow D_{sg}(\Lambda)$
is fully faithful. For any $X\in D_{sg}(\Lambda)$, there is a triangle
$$\xymatrix{ \widetilde{i}_*\widetilde{i}^!X \ar[r]^\eta & X
\ar[r]^\varepsilon & \widetilde{j}_*\widetilde{j}^*X \ar[r]^\omega & \widetilde{i}_*\widetilde{i}^!X[1]}$$
in $D_{sg}(\Lambda)$.
Since $\omega\in \Hom_{D_{sg}(\Lambda) }(\Im \widetilde{j}_*, \Im \widetilde{i}_*)=0$, we get that
$X\cong \widetilde{i}_*\widetilde{i}^!X\oplus
\widetilde{j}_*\widetilde{j}^*X$ in $D_{sg}(\Lambda)$; see \cite[1.4 Lemma]{Hap88} or
\cite[Proposition 4.8]{Miy00}. Therefore, the functor
$(\widetilde{i}_*, \widetilde{j}_*):D_{sg}(A)\coprod D_{sg}(B)\rightarrow D_{sg}(\Lambda)$
is dense, and then it induces a triangle equivalence between $D_{sg}(\Lambda)$ and $D_{sg}(A)\coprod D_{sg}(B)$.
\end{proof}

\section{Applications and examples}
\indent\indent
In this section, we will apply our main result to algebras given by quivers with relations.
Let $k$ be an algebraically closed field. Then any finite dimensional $k$-algebra
is isomorphic to $kQ/I$, where $Q$ is a finite quiver and $I$ is an admissible ideal.
A quiver $Q$ is a quadruples $(Q_0, Q_1, s, t)$, where $Q_0$ and $Q_1$ are finite sets of vertices and
arrows, $s$ and $t$ are functions sending any arrow in $Q_1$ to its starting point and ending point,
respectively. For arbitrary two arrows $\alpha _1$ and $\alpha _2$ with $t(\alpha _1) = s(\alpha _2)$, the composition of
$\alpha _1$ and $\alpha _1$ is denoted by $\alpha _1 \alpha _2$.

\begin{corollary}\label{cor-quiver}
Let $A = kQ_A/I_A$ and $B = kQ_B/I_B$ be two finite-dimensional
algebras, $(Q_\Lambda)_0 = (Q_A)_0 \cup (Q_B)_0$, and $(Q_\Lambda)_1 = (Q_A)_1 \cup (Q_B)_1 \cup \{\alpha _i  \}_{i\in J}$ where
$J$ is a finite subset of $\mathbb{N}$ and $\alpha _i$ is an arrow from $(Q_B)_0$ to $(Q_A)_0$ for any $i\in J$.
Let $\Lambda =kQ_\Lambda/I_\Lambda$ with $I_\Lambda$ being the ideal
generated by $I_A$, $I_B$ and $\alpha \alpha _i$, for any $\alpha \in (Q_B)_1$
and $i\in J$. Then $D_{sg}(\Lambda) \cong D_{sg}(A)\coprod D_{sg}(B).$
\end{corollary}
\begin{proof}
Since there are no arrows from $(Q_A)_0$ to $(Q_B)_0$, we have that $\Lambda =
\left[\begin{array}{cc} A & 0 \\ M & B  \end{array}\right] $ with $M=e_B\Lambda e_A$
being a $B$-$A$-bimodule, where
$e_A =
\left[\begin{array}{cc} 1_A & 0 \\ 0& 0  \end{array}\right] $ and $e_B=
\left[\begin{array}{cc} 0 & 0 \\ 0& 1_B  \end{array}\right] $.
To get the desired assertion, in view of Theorem~\ref{theorem-main}, it suffices to show that
$_BM$ is semi-simple and $M_A$ is projective. These facts
were mentioned in \cite[Section 1]{BM22} under the case of Morita context algebra, which is an extension of triangular matrix algebra.
Here, we give a proof for readers' convenience.

First, note that $\rad (_BM)=(\rad B) M=(\rad B)e_B\Lambda e_A$,
which is equal to zero because $\alpha \alpha _j\in I_\Lambda$, for any $\alpha \in (Q_B)_1$
and $i\in J$. Hence we get that $\rad (_BM)=0$ and thus $_BM$ is semi-simple.

For any $i\in J$, let $t_i:=t(\alpha _i)$ and let $e_{t_i}$ be the idempotent
of $A$ corresponding to the vertex $t_i$.
Let $p_{i1}+I_A$, $\cdots$, $p_{ir_i}+I_A$ be a basis of $e_{t_i}A$, and let $f_i:
e_{t_i}A \rightarrow M=e_B\Lambda e_A$  be the morphism of right $A$-modules mapping
$p_{ij}+I_A$ to $\alpha_ip_{ij}+I_\Lambda$, for any $j=1, \cdots, r_i$. Since $I_A \subseteq I_\Lambda$
, we get that $f_i$ is well-defined. Now we will claim
$f=(f_i)_{i\in J}:\
\oplus _{i\in J} (e_{t_i}A) \rightarrow M$ is an isomorphism of right $A$-modules,
and then
$M_A$ is projective.  Let $q$ be a path in
$Q_\Lambda$ from $(Q_B)_0$ to $(Q_A)_0$, which is not in $I_\Lambda$. Since $(\rad B)e_B\Lambda e_A=0$,
there is some $i\in J$ such
that $q=\alpha _ip$, where $p$ is a path in $Q_A$ starting form $t_i$. So $f$ is surjective.
For any $(\sum _{j=1}^{r_i}a_{ij}p_{ij}+I_A) _{i\in J}\in \oplus _{i\in J} (e_{t_i}A)$
with $a_{ij}\in k$,
if $f((\sum _{j=1}^{r_i}a_{ij}p_{ij}+I_A) _{i\in J})=0$, then $\sum _{i\in J}\sum _{j=1}^{r_i}
a_{ij}\alpha_{i}p_{ij}\in I_{\Lambda}$. Note that $I_{\Lambda}=\langle I_A, I_B, \mathcal{S}\rangle $,
where $\mathcal{S}=\{\alpha \alpha _i | i\in J \ \mbox{and} \  \alpha \in (Q_B)_1 \}$.
So there exist $u\in \langle I_A\rangle$, $v\in \langle I_B\rangle$
and $w\in \langle \mathcal{S} \rangle$ such that $\sum _{i\in J}\sum _{j=1}^{r_i}
a_{ij}\alpha_{i}p_{ij}=u+v+w$. Since $u,v$ and $w$ are linear combinations of paths
from $(Q_B)_0$ to $(Q_A)_0$, we may assume that $u=\sum _{i\in J}k_i\alpha _iu_i$,
$v=\sum _{i\in J}b_i v_i\alpha _i p_i$ and $w=\sum _{i\in J}c_i\beta _i\alpha _iq_i$,
where $k_i, b_i, c_i\in k$, $u_i\in I_A$, $v_i\in I_B$, $\beta _i\in (Q_B)_1$,
and $p_i$, $q_i$ are paths in $Q_A$ starting from $t_i$. Therefore,
$\sum _{i\in J}\sum _{j=1}^{r_i}
a_{ij}\alpha_{i}p_{ij}=\sum _{i\in J}k_i\alpha _iu_i+\sum _{i\in J}b_i v_i\alpha _i p_i+\sum _{i\in J}c_i\beta _i\alpha _iq_i$
in $kQ_\Lambda$. Since all paths form a basis of $kQ_\Lambda$, we get that
$b_i=0=c_i$ for any $i\in J$, and $\sum _{i\in J}\sum _{j=1}^{r_i}
a_{ij}\alpha_{i}p_{ij}=\sum _{i\in J}k_i\alpha _iu_i$. Then we have $\sum _{i\in J}\alpha_{i}(\sum _{j=1}^{r_i}
a_{ij}p_{ij})=\sum _{i\in J}\alpha _i(k_iu_i)$ in $kQ$,
which implies that $\sum _{j=1}^{r_i}
a_{ij}p_{ij}=k_iu_i\in I_A$, for any $i\in J$. In conclusion,
we get that $(\sum _{j=1}^{r_i}a_{ij}p_{ij}+I_A) _{i\in J}=0$,
and then $f$ is injective.
\end{proof}

\begin{example}\label{exam-1}
{\rm Let $\Lambda = kQ/I$ be the algebra where $Q$ is the quiver
$$\xymatrix{1 \ar@(ul,dl)_\alpha & 2 \ar[l]_\beta \ar@(ur,dr)^\gamma}$$
and $I=\langle \alpha ^3, \gamma ^2, \gamma \beta\rangle$. Let $A=e_1\Lambda e_1=k[x]/\langle x^3\rangle$
and $B=e_2\Lambda e_2=k[x]/\langle x^2\rangle$. Then both $A$ and $B$ are selfinjective algebras,
and the algebras $A$, $B$ and $\Lambda$ satisfy all the conditions in Corollary~\ref{cor-quiver}.
So $D_{sg}(\Lambda) \cong D_{sg}(A)\coprod D_{sg}(B)\cong \underline{\mod} (k[x]/\langle x^3\rangle)\coprod \mod k$.
We remark that $\Lambda$ is a monomial algebra, and it follows from
\cite[Theorem 4.1]{CSZ18} that there are only
two indecomposable objects in the stable category
$\underline{\Gproj }\Lambda$.
However, $\Lambda$ is non-Gorenstein and thus
the embedding $\underline{\Gproj }\Lambda \hookrightarrow D_{sg}(\Lambda)$
is not an equivalence.
Indeed, there are three indecomposable objects in $D_{sg}(\Lambda)$ by our calculation.
}
\end{example}

\begin{remark}\label{remark-1}
{\rm In general, the singular equivalence in Corollary~\ref{cor-quiver} can not induces
an equivalence on Gorenstein defect category. Consider the above example.
The Gorenstein defect categories of $A$ and $B$ are zero, but this is not true for $\Lambda$
since $\Lambda$ is non-Gorenstein. Therefore, the Gorenstein
defect category of $\Lambda$ is not equivalent to the direct sum of those of $A$ and $B$.
However, if we replace
$I=\langle \alpha ^3, \gamma ^2, \gamma \beta\rangle$ by $I=\langle \alpha ^3, \gamma ^2\rangle$,
then we get a singular equivalence $D_{sg}(\Lambda) \cong D_{sg}(A)\coprod D_{sg}(B)$
which induces
an equivalence on the corresponding Gorenstein defect categories, see
\cite[Corollary 4.4]{EPS22} and \cite[Corollary 4.1]{Qin22}.

}
\end{remark}

\noindent {\footnotesize {\bf ACKNOWLEDGMENT.} This work is supported by
the National Natural Science Foundation of China (12061060, 11961007),
the project of Young and Middle-aged Academic and Technological leader of Yunnan
(Grant No. 202305AC160005) and
the Scientific and Technological Innovation Team of Yunnan(Grant No.
2020CXTD25).}

\end{document}